\documentclass[11pt]{article}

\usepackage[T1]{fontenc}
\usepackage{lmodern}
\usepackage{amsmath,amssymb,amsthm}
\usepackage{microtype}
\usepackage[margin=1.15in]{geometry}
\usepackage[hidelinks]{hyperref}

\newtheorem{theorem}{Theorem}[section]
\newtheorem{proposition}[theorem]{Proposition}
\newtheorem{lemma}[theorem]{Lemma}

\newcommand{\FF}{\mathcal F}
\newcommand{\EE}{\mathcal E}
\newcommand{\LL}{\mathcal L}
\newcommand{\TT}{\mathcal T}
\newcommand{\eps}{\varepsilon}

\title{Fibonacci, Dirichlet, and Gauss in a single sum}
\author{Beno\^it Cloitre\\[2pt]
\small\href{https://orcid.org/0009-0001-6778-153X}{ORCID: 0009-0001-6778-153X}}
\date{}

\begin{document}

\maketitle

\begin{abstract}
We study the fractional-part sums $\sum_{k=1}^{n}\{F_n/F_k\}$, where $F_n$
is the $n$th Fibonacci number. Their asymptotic behavior depends on the
parity of $n$. For odd $n$, the remainder is expressed in terms of the
Gauss circle error term. For even $n$, it is expressed in terms of the
Dirichlet divisor error term. Thus determining the optimal remainder exponent
for the odd Fibonacci sums is equivalent to the Gauss circle problem, while
the corresponding question for the even sums is equivalent to the Dirichlet
divisor problem. We also prove analogous formulas for a family of second-order
recurrences, including the Lucas sequence, for which the roles of the two
parities are exchanged.
\end{abstract}

\noindent
\emph{Keywords.} Fibonacci numbers, Lucas numbers, fractional parts,
Dirichlet divisor problem, Gauss circle problem.

\smallskip
\noindent
\emph{2020 mathematics subject classification.} 11B39, 11N37, 11P21.

\section{Introduction}

The Fibonacci sequence $(F_j)_{j\geq 0}$ is defined by
\[
 F_0=0,\qquad F_1=1,\qquad F_{j+2}=F_{j+1}+F_j
 \quad (j\geq 0).
\]
For every $n\in\mathbb N=\{1,2,3,\ldots\}$, we study the sum
\begin{equation}\label{eq:def-F-sum}
 \FF(n)=\sum_{k=1}^{n}\left\{\frac{F_n}{F_k}\right\},
\end{equation}
where $\{x\}=x-\lfloor x\rfloor$ is the fractional part of the real
number $x$ and $\lfloor x\rfloor$ is the greatest integer not exceeding
$x$. Throughout the paper, $\log$ denotes the natural logarithm. The sum
\eqref{eq:def-F-sum} arose in a broader study of Tauberian and Abelian
sums in Part~IV of the author's monograph in preparation \cite{CloitreRAF}.

The asymptotic behavior of \eqref{eq:def-F-sum} brings together two
classical lattice-point problems. Write $\#S$ for the number of elements
of a finite set $S$. For every real number $x\geq1$, define
\begin{align*}
 N_H(x)&=\#\{(a,b)\in\mathbb N^2:ab\leq x\},\\
 N_C(x)&=\#\{(a,b)\in\mathbb Z^2:a^2+b^2\leq x\}.
\end{align*}
The function $N_H(x)$ counts positive lattice points on or below the
hyperbola $ab=x$, while $N_C(x)$ counts integer lattice points inside or on
the circle $a^2+b^2=x$. These geometric counts also have arithmetic forms. Let
$\tau(j)$ denote the number of positive divisors of $j$. Counting the points
under the hyperbola first by their first coordinate, and then by their product,
gives
\[
 N_H(x)=\sum_{k\leq x}\left\lfloor\frac{x}{k}\right\rfloor
       =\sum_{j\leq x}\tau(j).
\]
Jacobi's two-square identity gives the corresponding floor sum for the circle
\cite{HardyCircle,Hirschhorn},
\[
 N_C(x)=1+4\sum_{m\geq0}
 \left(
  \left\lfloor\frac{x}{4m+1}\right\rfloor
  -\left\lfloor\frac{x}{4m+3}\right\rfloor
 \right).
\]
The two residue classes modulo $4$ in this identity reappear in the
odd-indexed Fibonacci sum.

Let $\gamma$ denote Euler's constant. Define the error terms $\Delta_H$ and
$\Delta_C$ by
\begin{align}
 N_H(x)&=x\log x+(2\gamma-1)x+\Delta_H(x),
      \label{eq:def-hyperbola}\\
 N_C(x)&=\pi x+\Delta_C(x).
      \label{eq:def-circle}
\end{align}
These are the standard normalizations for the two problems
\cite{Dirichlet,Guy,BerndtKimZaharescu}. The Dirichlet divisor problem asks
how small the exponent $\theta_H$ can be in the bound
\[
 \Delta_H(x)=O(x^{\theta_H+\eps})
 \qquad(\eps>0),
\]
and the Gauss circle problem asks the same question for the exponent
$\theta_C$ in
\[
 \Delta_C(x)=O(x^{\theta_C+\eps})
 \qquad(\eps>0).
\]
Dirichlet proved $\Delta_H(x)=O(x^{1/2})$ \cite{Dirichlet}, and Vorono\"i
proved $\Delta_H(x)=O(x^{1/3}\log x)$ \cite{Voronoi}. Gauss's elementary
geometric argument gives $\Delta_C(x)=O(x^{1/2})$
\cite{Guy,BerndtKimZaharescu}. Hardy's omega results show that neither error
term is $o(x^{1/4})$ \cite{HardyCircle,HardyDivisor}. The conjectured optimal
bounds are
\[
 \Delta_H(x)=O(x^{1/4+\eps}),\qquad
 \Delta_C(x)=O(x^{1/4+\eps})
 \qquad(\eps>0).
\]
Equivalently, the conjectural optimal exponents are
$\theta_H=\theta_C=1/4$. In this literature, the functions $\Delta_H$ and
$\Delta_C$ are usually denoted by $\Delta$ and $P$, respectively.

The two problems have long been studied in parallel. Ingham treated them
together \cite{Ingham}, and Iwaniec and Mozzochi bounded both error terms
by a common method \cite{IwaniecMozzochi}. The analogy also appears in the
corresponding Vorono\"i formulas for the two error terms \cite{Ingham,Ivic}.
Huxley proved that the exponent $131/416$ is admissible in both problems
\cite{Huxley}, and a 2023 preprint of Li and Yang claims the smaller common
admissible exponent
$0.3144831759741\ldots$ \cite{LiYang}. For broader accounts, see Ivi\'c
\cite{Ivic}, Huxley \cite{HuxleyBook,HuxleyMillennium}, Guy \cite{Guy}, and
Berndt, Kim, and Zaharescu \cite{BerndtKimZaharescu}. 

The following theorem is the main result of this paper. It places both
classical error terms inside the single Fibonacci sum
\eqref{eq:def-F-sum}, with the parity of $n$ determining which one appears.

\begin{theorem}\label{thm:main}
For every odd integer $n\geq 3$ and every $\eps>0$, one has
\begin{equation}\label{eq:odd-main}
 \FF(n)=\frac{\pi}{8}n
 +\frac14\left(\Delta_C(n)-\Delta_C\left(\frac n2\right)\right)
 +O(n^\eps).
\end{equation}
For every even integer $n\geq 4$ and every $\eps>0$, one has
\begin{equation}\label{eq:even-main}
 \FF(n)=\frac{3\log 2}{4}n
 +2\Delta_H\left(\frac n2\right)
 -5\Delta_H\left(\frac n4\right)
 +2\Delta_H\left(\frac n8\right)
 +O(n^\eps).
\end{equation}
\end{theorem}

A question about the even-index limit had previously appeared on
MathOverflow. In response, GH from MO gave a short proof of the two
parity-dependent leading terms with an error $O(\sqrt n)$ \cite{GHMO}.
Theorem~\ref{thm:main} refines that result by expressing both remainders in
terms of the two classical error functions, up to $O(n^\eps)$.

The relation between the sum and the two classical problems is most simply
stated in terms of error exponents. Let $0\leq\theta_C,\theta_H<1$.
Theorem~\ref{thm:main} and proposition~\ref{prop:converse} give
\begin{equation}\label{eq:exponent-equivalence}
 \left.
 \begin{aligned}
  \Delta_C(x)&=O(x^{\theta_C+\eps}),\\
  \Delta_H(x)&=O(x^{\theta_H+\eps})
 \end{aligned}
 \right\}
 \quad\Longleftrightarrow\quad
 \FF(n)=
 \begin{cases}
  \dfrac{\pi}{8}n+O(n^{\theta_C+\eps}),&n\text{ odd},\\[6pt]
  \dfrac{3\log 2}{4}n+O(n^{\theta_H+\eps}),&n\text{ even}
 \end{cases}
 \qquad(\eps>0).
\end{equation}
The equivalence is unconditional. In particular, the conjecture that
$1/4$ is the optimal exponent in both classical problems is equivalent
to the conjecture that $1/4$ is the optimal remainder exponent for
$\FF(n)$, separately along the odd and even indices.

The proof combines Fibonacci congruences with exponentially weighted
divisor sums. Section~2 establishes the quotient identity used throughout.
Sections~3, 4, and~5 derive the periodic profiles and prove
theorem~\ref{thm:main}. Jacobi's two-square identity handles the odd indices,
while an alternating divisor identity handles the even indices. Section~6
proves the converse implications in \eqref{eq:exponent-equivalence}.
Section~7 treats generalized Fibonacci and Lucas sequences and records a
numerical observation for the Tribonacci sequence.

\section{A quotient identity and an exponential example}

We first record an identity for functions of $\lfloor x/k\rfloor$ and apply it
to an elementary exponential sequence. This example already produces the
Dirichlet divisor error term.

\begin{lemma}\label{lem:quotient}
Let $\omega$ be a complex-valued function on the nonnegative integers with
$\omega(0)=0$. For every real number $x\geq 1$, one has
\begin{equation}\label{eq:quotient-identity}
 \sum_{k\leq x}\omega\left(\left\lfloor\frac{x}{k}\right\rfloor\right)
 =\sum_{q\leq x}\bigl(\omega(q)-\omega(q-1)\bigr)
  \left\lfloor\frac{x}{q}\right\rfloor
 =\sum_{j\leq x}\sum_{q\mid j}\bigl(\omega(q)-\omega(q-1)\bigr).
\end{equation}
\end{lemma}

\begin{proof}
For every nonnegative integer $t$, telescoping gives
\[
 \omega(t)=\sum_{q\leq t}\bigl(\omega(q)-\omega(q-1)\bigr).
\]
Substitution on the left of \eqref{eq:quotient-identity}, followed by an
interchange of the two finite sums, gives the first equality. Grouping the
terms according to $j=kq$ gives the second.
\end{proof}

For integers $m\geq 2$ and $n\geq 1$, define
\begin{equation}\label{eq:def-exp-sum}
 \EE_m(n)=\sum_{k\leq n}
 \left\{\frac{m^n+1}{m^k+1}\right\}.
\end{equation}
\begin{proposition}\label{prop:exponential}
Let $m\geq 2$ and $n\geq 1$ be integers. Write
$n=2^\nu n_{\mathrm{odd}}$, where $\nu$ is a nonnegative integer and
$n_{\mathrm{odd}}$ is odd. For every $\eps>0$, one has
\begin{equation}\label{eq:exp-remainder}
 \EE_m(n)=n\log 2+\Delta_H(n)-2\Delta_H\left(\frac n2\right)
 -\tau(n_{\mathrm{odd}})+O(n^\eps).
\end{equation}
\end{proposition}

\begin{proof}
For each $k\leq n$, write $n=qk+r$ with $0\leq r<k$. Reduction modulo
$m^k+1$ gives
\[
 m^n+1\equiv
 \begin{cases}
  m^r+1& q\text{ even},\\
  1-m^r& q\text{ odd}
 \end{cases}
 \pmod {m^k+1}.
\]
When $q$ is odd and $r>0$, the fractional part in
\eqref{eq:def-exp-sum} equals
\[
 1-\frac{m^r-1}{m^k+1}.
\]
When $q$ is even, it equals $(m^r+1)/(m^k+1)$. The terms with odd $q$
and $r=0$ vanish.

The fractional part therefore differs from $1$ when $q$ is odd and $r>0$,
and from $0$ otherwise, by at most a constant times
$m^{-(k-r)}+m^{-k}$. If $h=k-r$, then $k$ divides $n+h$. The standard
divisor estimate $\tau(j)=O(j^\eps)$ gives
\[
 \sum_{k\leq n}m^{-(k-r)}
 \ll \sum_{h\geq 1}m^{-h}\tau(n+h)
 \ll n^\eps.
\]
Here the last bound is uniform in $m$, since $m\geq 2$. The terms $m^{-k}$
also have a uniformly bounded sum.

Apply lemma~\ref{lem:quotient} to the function that equals $1$ on odd
positive integers and $0$ otherwise. Its first difference is $(-1)^{q+1}$,
and the lemma gives
\[
 \#\left\{k\leq n: \left\lfloor\frac nk\right\rfloor
 \text{ is odd}\right\}
 =N_H(n)-2N_H\left(\frac n2\right).
\]
The excluded pairs with $r=0$ correspond to divisors $k$ of $n$ for
which $n/k$ is odd. Their number is $\tau(n_{\mathrm{odd}})$. Substitute
\eqref{eq:def-hyperbola} into the last identity to obtain
\eqref{eq:exp-remainder}.
\end{proof}

Proposition~\ref{prop:exponential} produces the divisor error term from a
sequence with a simple residue law. The Fibonacci sequence adds a parity
interaction that also produces the circle error term.

\section{Fibonacci residues}

For positive integers $n$ and $k$ with $k\leq n$, let $q$ and $r$ be the
quotient and the remainder in the division of $n$ by $k$. Thus
\[
 n=qk+r,\qquad 0\leq r<k.
\]

\begin{lemma}\label{lem:fib-congruence}
Let $k\geq 3$ and $s\geq 0$ be integers. For every integer $r$ with
$0\leq r<k$, one has
\begin{align}
 F_{2sk+r}&\equiv (-1)^{sk}F_r\pmod {F_k},\label{eq:even-q-cong}\\
 F_{(2s+1)k+r}&\equiv (-1)^{sk+r-1}F_{k-r}\pmod {F_k}
 \quad (r>0).\label{eq:odd-q-cong}
\end{align}
If $n=qk$ for a positive integer $q$, then $F_k$ divides $F_n$.
\end{lemma}

\begin{proof}
The first congruence is immediate when $s=0$. For every positive integer
$q$, the Fibonacci addition formula gives
\[
 F_{qk+r}=F_{qk-1}F_r+F_{qk}F_{r+1}.
\]
Consequently,
\begin{equation}\label{eq:basic-residue}
 F_{qk+r}\equiv F_{qk-1}F_r\pmod {F_k}.
\end{equation}
Set
\[
 Q=\begin{pmatrix}1&1\\1&0\end{pmatrix},\qquad
 Q^k=\begin{pmatrix}F_{k+1}&F_k\\F_k&F_{k-1}\end{pmatrix}.
\]
Reduction modulo $F_k$ gives
\[
 Q^k\equiv F_{k-1}\begin{pmatrix}1&0\\0&1\end{pmatrix}\pmod{F_k}.
\]
Taking the lower right entry of $Q^{qk}$ gives
\[
 F_{qk-1}\equiv F_{k-1}^{\,q}\pmod {F_k}.
\]
Cassini's identity gives
\[
 F_{k-1}^2\equiv (-1)^k\pmod {F_k}.
\]
These two congruences and \eqref{eq:basic-residue} prove
\eqref{eq:even-q-cong}. For odd $q$, use
\[
 F_{k-1}F_r\equiv (-1)^{r-1}F_{k-r}\pmod {F_k},
\]
which follows from the determinant identity
$F_{k-1}F_r-F_kF_{r-1}=(-1)^{r-1}F_{k-r}$. This proves
\eqref{eq:odd-q-cong}. The divisibility assertion is the standard relation
$F_a\mid F_b$ when $a\mid b$.
\end{proof}

The signs in lemma~\ref{lem:fib-congruence} determine whether a fractional
part is close to zero or to one. Define the periodic functions $w$, $u$, and
$v$ by the following table.
\[
\begin{array}{c|cccc}
q\bmod4&0&1&2&3\\ \hline
w(q)&0&1&1&0\\
u(q)&0&1&0&1\\
v(q)&0&0&1&1
\end{array}
\]

\begin{lemma}\label{lem:profiles}
For every odd integer $n\geq 3$ and every $\eps>0$, one has
\begin{equation}\label{eq:odd-profile}
 \FF(n)=\sum_{\substack{k\leq n\\ k\text{ odd}}}
 w\left(\left\lfloor\frac nk\right\rfloor\right)+O(n^\eps).
\end{equation}
For every even integer $n\geq 4$ and every $\eps>0$, one has
\begin{equation}\label{eq:even-profile}
 \FF(n)=\sum_{\substack{k\leq n\\ k\text{ even}}}
 u\left(\left\lfloor\frac nk\right\rfloor\right)
 +\sum_{\substack{k\leq n\\ k\text{ odd}}}
 v\left(\left\lfloor\frac nk\right\rfloor\right)
 +O(n^\eps).
\end{equation}
\end{lemma}

\begin{proof}
Terms with $k\leq 2$ contribute $O(1)$, so assume $k\geq 3$ and use
the division $n=qk+r$. If $q=2t$ and $r>0$, formula
\eqref{eq:even-q-cong} shows that the fractional part is exactly
$F_r/F_k$ or $1-F_r/F_k$, according as $tk$ is even or odd. It vanishes
when $r=0$. Thus a contribution close to one requires $k$ odd and
$q\equiv 2\pmod 4$.

If $q=2t+1$ and $r>0$, formula
\eqref{eq:odd-q-cong} shows that the fractional part is exactly
$F_{k-r}/F_k$ or $1-F_{k-r}/F_k$, according as $tk+r-1$ is even or odd.
If $n$ is odd, the second case reduces to $k$ odd and
$q\equiv 1\pmod 4$. If $n$ is even, it reduces to either $k$ even and $q$
odd, or $k$ odd and $q\equiv 3\pmod 4$. Combining the odd and even values
of $q$ gives the two formulas in the lemma.

Put $\varphi=(1+\sqrt5)/2$. Then
\[
 \frac{F_r}{F_k}\ll\varphi^{-(k-r)},\qquad
 \frac{F_{k-r}}{F_k}\ll\varphi^{-r}.
\]
For the first ratio, put $h=k-r$. The equality $n=qk+r$ implies
$k\mid n+h$. For the second ratio, put $h=r$, which gives $k\mid n-h$.
Hence the total error is
\[
 \ll\sum_{h\geq 1}\varphi^{-h}
 \bigl(\tau(n+h)+\tau(n-h)\bigr)+O(\tau(n)),
\]
where terms with nonpositive arguments are omitted. Split the sum at $h=n$.
For $h\leq n$, the divisor estimate $\tau(j)=O(j^\eps)$ applies, while for
$h>n$ the exponential weight absorbs the polynomial growth of
$\tau(n+h)$. This proves the result. The final term $O(\tau(n))$ accounts
for the pairs with $r=0$, which the periodic functions may count even though
the corresponding fractional parts vanish.
\end{proof}

\section{Odd indices and the circle problem}

For every positive integer $q$, define $\chi_4(q)$ by
\[
 \chi_4(q)=
 \begin{cases}
  0&q\text{ even},\\
  1&q\equiv 1\pmod 4,\\
  -1&q\equiv 3\pmod 4.
 \end{cases}
\]
This is the nonprincipal character modulo $4$. Since
\[
 w(q)-w(q-1)=\chi_4(q),
\]
lemma~\ref{lem:quotient} gives
\[
 \sum_{k\leq x}w\left(\left\lfloor\frac{x}{k}\right\rfloor\right)
 =\sum_{q\leq x}\chi_4(q)\left\lfloor\frac{x}{q}\right\rfloor.
\]

For every positive integer $j$, let $r_2(j)$ denote the number of ordered
pairs $(p_1,p_2)\in\mathbb Z^2$ satisfying $j=p_1^2+p_2^2$. Jacobi's two-square
identity states that
\[
 r_2(j)=4\sum_{e\mid j}\chi_4(e).
\]
Summing over $j\leq x$ and using the preceding identity gives
\begin{equation}\label{eq:w-circle}
 \sum_{k\leq x}w\left(\left\lfloor\frac{x}{k}\right\rfloor\right)
 =\frac{N_C(x)-1}{4}.
\end{equation}
In \eqref{eq:w-circle} with $x=n$, the even values $k=2j$ contribute
$\sum_{j\leq n/2}w(\lfloor n/(2j)\rfloor)$, and $n/(2j)=(n/2)/j$, so this
contribution equals $(N_C(n/2)-1)/4$ by the same identity at $x=n/2$.
Subtracting it and using \eqref{eq:odd-profile} gives
\[
 \FF(n)=\frac14\left(N_C(n)-N_C\left(\frac n2\right)\right)
 +O(n^\eps)
\]
for odd $n$. Formula \eqref{eq:odd-main} then follows from
\eqref{eq:def-circle}.

The two-dimensional count does not enter through an estimate. It follows from
Jacobi's arithmetic identity. The parity of the Fibonacci residue selects
$\chi_4$, the function whose divisor sum appears in the two-square formula.

\section{Even indices and the divisor problem}

For every real $x\geq 1$, define
\[
 U(x)=\sum_{k\leq x}u\left(\left\lfloor\frac{x}{k}\right\rfloor\right).
\]
The first difference of $u$ is $(-1)^{q+1}$. Lemma~\ref{lem:quotient}
therefore gives
\begin{equation}\label{eq:U}
 U(x)=N_H(x)-2N_H\left(\frac x2\right).
\end{equation}
The first difference of $v$ vanishes on odd integers and equals
$(-1)^{j+1}$ at $q=2j$ for every positive integer $j$. Thus
\[
 \sum_{k\leq x}v\left(\left\lfloor\frac{x}{k}\right\rfloor\right)
 =U\left(\frac x2\right).
\]

In \eqref{eq:even-profile}, the sum over the even values $k=2j$ is
$U(n/2)$, since $n/(2j)=(n/2)/j$. The sum over odd $k$ involves $v$. The
$v$-sum over all $k\leq n$ equals $U(n/2)$, and its even values $k=2j$
contribute $U(n/4)$, so the odd values contribute $U(n/2)-U(n/4)$.
Equation \eqref{eq:U} yields
\begin{align*}
 \FF(n)
 &=2U\left(\frac n2\right)-U\left(\frac n4\right)+O(n^\eps)
 \\
 &=2N_H\left(\frac n2\right)-5N_H\left(\frac n4\right)
 +2N_H\left(\frac n8\right)+O(n^\eps).
\end{align*}
Substitution of \eqref{eq:def-hyperbola} cancels the $n\log n$ and
$(2\gamma-1)n$ terms. The remaining linear
term is $3n\log 2/4$. This proves \eqref{eq:even-main} and completes the
proof of theorem~\ref{thm:main}.

\section{Equivalence of error exponents}

Theorem~\ref{thm:main} gives the forward implications in
\eqref{eq:exponent-equivalence}. The converse implications follow from the
next proposition.

\begin{proposition}\label{prop:converse}
Fix a real number $\theta$ with $0\leq\theta<1$.
Suppose that, for every $\eps>0$,
\[
 \FF(n)-\frac{\pi}{8}n=O(n^{\theta+\eps})
 \qquad (n\text{ odd}).
\]
Then $\Delta_C(x)=O(x^{\theta+\eps})$ for every $\eps>0$.
Suppose instead that, for every $\eps>0$,
\[
 \FF(n)-\frac{3\log 2}{4}n=O(n^{\theta+\eps})
 \qquad (n\text{ even}).
\]
Then $\Delta_H(x)=O(x^{\theta+\eps})$ for every $\eps>0$.
\end{proposition}

\begin{proof}
Assume first the bound along odd indices. Formula \eqref{eq:odd-main}, applied
with $n=2m+1$, gives
\[
 \Delta_C(2m+1)-\Delta_C\left(m+\frac12\right)
 =O(m^{\theta+\eps}).
\]
Since a sum of two integer squares is an integer,
$N_C(m+1/2)=N_C(m)$. Hence
$\Delta_C(m+1/2)=\Delta_C(m)-\pi/2$. Jacobi's identity and the
divisor bound also give
\[
 \Delta_C(j)-\Delta_C(j-1)=r_2(j)-\pi=O(j^\eps).
\]
These estimates yield, for every integer $j\geq 2$,
\[
 \Delta_C(j)=\Delta_C\left(\left\lfloor\frac j2\right\rfloor\right)
 +O(j^{\theta+\eps}).
\]
Iteration gives $\Delta_C(j)=O(j^{\theta+\eps})$. The same estimate holds
for real $x$, since $N_C(x)=N_C(\lfloor x\rfloor)$.

Assume next the bound along even indices. Formula \eqref{eq:even-main}, with
$n=8m$, gives
\[
 2\Delta_H(4m)-5\Delta_H(2m)+2\Delta_H(m)
 =O(m^{\theta+\eps}).
\]
For every positive integer $m$, set
\[
 B(m)=\Delta_H(2m)-2\Delta_H(m).
\]
In terms of $B$, the preceding bound is
\[
 2B(2m)-B(m)=O(m^{\theta+\eps}).
\]
The divisor bound and \eqref{eq:def-hyperbola} imply
\[
 \Delta_H(m+1)-\Delta_H(m)=O(m^\eps),
\]
so that $B(m+1)-B(m)=O(m^\eps)$. The two bounds for $B$ combine to give
\[
 B(j)=\frac12B\left(\left\lfloor\frac j2\right\rfloor\right)
 +O(j^{\theta+\eps})
\]
for every integer $j\geq 2$. Binary iteration shows that
\begin{equation}\label{eq:B-bound}
 B(j)=O(j^{\theta+\eps}).
\end{equation}

Fix $\eps>0$, and choose $\eta$ with
$0<\eta<\eps$ and $\theta+\eta<1$. For every positive integer $J$, the
definition of $B$ gives
\[
 \frac{\Delta_H(m)}m
 =\frac{\Delta_H(2^Jm)}{2^Jm}
 -\sum_{j=0}^{J-1}\frac{B(2^jm)}{2^{j+1}m}.
\]
The first term tends to zero as $J$ tends to infinity, since the known bounds
for the divisor error term imply $\Delta_H(x)=o(x)$. Estimate
\eqref{eq:B-bound}, used with the exponent $\theta+\eta$, makes the series
convergent and gives
\[
 \Delta_H(m)=O(m^{\theta+\eta})=O(m^{\theta+\eps}).
\]
This proves the required bound at integer arguments. Finally,
$N_H(x)=N_H(\lfloor x\rfloor)$, and the remaining main term varies by
$O(\log x)$ on an interval of length $1$. This extends the estimate from
integers to real $x$.
\end{proof}

This equivalence suggests a question for a wider family of fractional-part
sums.
Given an increasing sequence $A=(A_n)$ of positive integers, determine when
there is a periodic function $\omega$ on the positive integers, extended by
$\omega(0)=0$, such that, for every $\eps>0$,
\[
 \sum_{k\leq n}\left\{\frac{A_n}{A_k}\right\}
 =\sum_{k\leq n}
 \omega\left(\left\lfloor\frac nk\right\rfloor\right)
 +O(n^\eps).
\]
Lemma~\ref{lem:quotient} then turns the first difference of $\omega$ into a
divisor sum. The examples above yield the divisor function, its alternating
analogue, and the character $\chi_4$. Other residue laws may select different
divisor sums or representation functions.
Establishing such a law requires congruence information about the sequence,
not growth alone.

\section{Lucas numbers and a second-order extension}

The residue calculation extends to a one-parameter family of second-order
recurrences and to the associated Lucas sequences. Fix a positive integer
$c$. Define the sequences
$(F_j^{(c)})_{j\geq 0}$ and
$(L_j^{(c)})_{j\geq 0}$ by
\begin{align*}
 F_0^{(c)}&=0,&F_1^{(c)}&=1,&
 F_{j+2}^{(c)}&=cF_{j+1}^{(c)}+F_j^{(c)},\\
 L_0^{(c)}&=2,&L_1^{(c)}&=c,&
 L_{j+2}^{(c)}&=cL_{j+1}^{(c)}+L_j^{(c)}
 \qquad (j\geq 0).
\end{align*}
For $c=1$, these are the Fibonacci and Lucas sequences. For every positive
integer $n$, put
\[
 \FF_c(n)=\sum_{k=1}^{n}\left\{\frac{F_n^{(c)}}{F_k^{(c)}}\right\},
 \qquad
 \LL_c(n)=\sum_{k=1}^{n}\left\{\frac{L_n^{(c)}}{L_k^{(c)}}\right\}.
\]

\begin{theorem}\label{thm:second-order}
Let $c$ be a positive integer and let $\eps>0$. For every odd integer
$n\geq 3$, one has
\begin{align}
 \FF_c(n)&=\frac{\pi}{8}n
 +\frac14\left(\Delta_C(n)-\Delta_C\left(\frac n2\right)\right)
 +O(n^\eps),\label{eq:general-F-odd}\\
 \LL_c(n)&=\frac{3\log 2}{4}n
 +\Delta_H(n)-3\Delta_H\left(\frac n2\right)
 +3\Delta_H\left(\frac n4\right)-2\Delta_H\left(\frac n8\right)
 +O(n^\eps).\label{eq:general-L-odd}
\end{align}
For every even integer $n\geq 4$, one has
\begin{align}
 \FF_c(n)&=\frac{3\log 2}{4}n
 +2\Delta_H\left(\frac n2\right)-5\Delta_H\left(\frac n4\right)
 +2\Delta_H\left(\frac n8\right)+O(n^\eps),
 \label{eq:general-F-even}\\
 \LL_c(n)&=\frac{\pi}{8}n
 +\frac14\Delta_C\left(\frac n2\right)+O(n^\eps).
 \label{eq:general-L-even}
\end{align}
\end{theorem}

In particular, the ordinary Lucas sequence reverses the Fibonacci parity
pattern. Its odd indices select the Dirichlet divisor error term, while its
even indices select the Gauss circle error term.

\begin{proof}
Set
\[
 M_c=\begin{pmatrix}c&1\\1&0\end{pmatrix}.
\]
For every positive integer $k$, one has
\[
 M_c^k=
 \begin{pmatrix}
  F_{k+1}^{(c)}&F_k^{(c)}\\
  F_k^{(c)}&F_{k-1}^{(c)}
 \end{pmatrix}
 \equiv F_{k-1}^{(c)}
 \begin{pmatrix}1&0\\0&1\end{pmatrix}
 \pmod {F_k^{(c)}}.
\]
The determinant of $M_c$ is $-1$. For $1\leq r<k$, Cassini's identity and
the determinant identity used in lemma~\ref{lem:fib-congruence} take the forms
\begin{align*}
 F_{k+1}^{(c)}F_{k-1}^{(c)}-(F_k^{(c)})^2&=(-1)^k,\\
 F_{k-1}^{(c)}F_r^{(c)}-F_k^{(c)}F_{r-1}^{(c)}
 &=(-1)^{r-1}F_{k-r}^{(c)}.
\end{align*}
The congruences in that lemma consequently hold with $F_j$ replaced by
$F_j^{(c)}$.

The corresponding Lucas congruences take the following form. Let $k\geq 2$,
$s\geq 0$, and $0\leq r<k$. Then
\begin{align}
 L_{2sk+r}^{(c)}&\equiv
 (-1)^{s(k+1)}L_r^{(c)}\pmod {L_k^{(c)}},
 \label{eq:Lucas-even-q}\\
 L_{(2s+1)k+r}^{(c)}&\equiv
 (-1)^{s(k+1)+r+1}L_{k-r}^{(c)}\pmod {L_k^{(c)}}
 \quad (r>0).
 \label{eq:Lucas-odd-q}
\end{align}
To prove these formulas, let $\alpha$ and $\beta$ be the roots of
$z^2-cz-1$. Thus $\alpha\beta=-1$ and
$L_j^{(c)}=\alpha^j+\beta^j$. For fixed $k$ and $r$, the terms
$L_{qk+r}^{(c)}$, viewed as a sequence in $q$, satisfy
\[
 L_{(q+2)k+r}^{(c)}
 =L_k^{(c)}L_{(q+1)k+r}^{(c)}-(-1)^kL_{qk+r}^{(c)}.
\]
Reduction modulo $L_k^{(c)}$ and iteration over the even values of $q$ give
\eqref{eq:Lucas-even-q}. For odd values of $q$, use
\[
 L_{k+r}^{(c)}=L_k^{(c)}L_r^{(c)}-(-1)^rL_{k-r}^{(c)},
\]
which proves \eqref{eq:Lucas-odd-q}. When $r=0$ and the quotient is odd,
$L_k^{(c)}$ divides the numerator.

The dominant root is
\[
 \alpha=\frac{c+\sqrt{c^2+4}}{2}>1,
\]
while $|\beta|=\alpha^{-1}$. Since $c\geq 1$, one has $\alpha\geq\varphi$.
The estimates in the proof of lemma~\ref{lem:profiles} remain valid for the
ratios of generalized Fibonacci and Lucas numbers, with $\varphi$ replaced
by $\alpha$. Since $\alpha\geq\varphi$, their constants may be chosen
independently of $c$, and the total error is $O(n^\eps)$.

The generalized Fibonacci congruences give the same two profiles as
\eqref{eq:odd-profile} and \eqref{eq:even-profile}. Hence
\eqref{eq:general-F-odd} and \eqref{eq:general-F-even} follow from the
proof of theorem~\ref{thm:main}.

For the generalized Lucas sequence, \eqref{eq:Lucas-even-q} and
\eqref{eq:Lucas-odd-q} give
\begin{align*}
 \LL_c(n)&=
 \sum_{\substack{k\leq n\\k\text{ even}}}
 w\left(\left\lfloor\frac nk\right\rfloor\right)
 +O(n^\eps)
 &&(n\text{ even}),\\
 \LL_c(n)&=
 \sum_{\substack{k\leq n\\k\text{ odd}}}
 u\left(\left\lfloor\frac nk\right\rfloor\right)
 +\sum_{\substack{k\leq n\\k\text{ even}}}
 v\left(\left\lfloor\frac nk\right\rfloor\right)
 +O(n^\eps)
 &&(n\text{ odd}).
\end{align*}
Writing $k=2j$ in the first profile and applying \eqref{eq:w-circle} at
$x=n/2$ gives
\[
 \LL_c(n)=\frac14\left(N_C\left(\frac n2\right)-1\right)
 +O(n^\eps).
\]
For odd $n$, the $u$-sum over all $k\leq n$ equals $U(n)$ and its even
values $k=2j$ contribute $U(n/2)$, so the sum over odd $k$ is
$U(n)-U(n/2)$. The sum over the even values $k=2j$ involving $v$ is
$U(n/4)$. Equation \eqref{eq:U} therefore gives
\begin{align*}
 \LL_c(n)
 &=U(n)-U\left(\frac n2\right)+U\left(\frac n4\right)
 +O(n^\eps)\\
 &=N_H(n)-3N_H\left(\frac n2\right)+3N_H\left(\frac n4\right)
 -2N_H\left(\frac n8\right)+O(n^\eps).
\end{align*}
Substitution of \eqref{eq:def-hyperbola} and \eqref{eq:def-circle} proves
\eqref{eq:general-L-odd} and \eqref{eq:general-L-even}.
\end{proof}

Theorem~\ref{thm:second-order} does not cover every recurrence of order two
with integer coefficients. Its proof uses both the determinant $-1$ of $M_c$
and the chosen initial values. If the coefficient of the older term
differs from $1$, a matrix power need not reduce to a scalar sign. Arbitrary
initial values can also destroy the divisibility relations. The residue law
modulo the denominator, rather than the order of the recurrence, governs the
resulting fractional-part sum.

The first higher-order case may behave differently. Define the
Tribonacci sequence by
\[
 T_1=1,\qquad T_2=1,\qquad T_3=2,\qquad
 T_{j+3}=T_{j+2}+T_{j+1}+T_j\quad (j\geq 1),
\]
and put
\[
 \TT(n)=\sum_{k=1}^{n}\left\{\frac{T_n}{T_k}\right\}.
\]
Direct computations using the exact remainders $T_n\bmod T_k$, for
$n\leq 4000$, suggest that $\TT(n)/n$ is close to $1/2$. The three residue
classes of $n$ modulo $3$ have the same apparent leading coefficient over
this range. A trichotomy could still occur in the remainder, or depend on
$\lfloor n/k\rfloor$ rather than on $n$. No asymptotic formula is claimed
here.

The Tribonacci companion matrix is $3\times3$. A direct reduction modulo
$T_k$ involves several residues rather than a scalar sign. Any analogue of the
quotient profiles above must account for this additional residue data. Such
profiles may lead to counting functions other than those in the Dirichlet
divisor and Gauss circle problems. It remains open whether
$\TT(n)=n/2+o(n)$ and what arithmetic object, if any, controls the remainder.

\end{document}